\documentclass{amsart}
\pdfoutput=1

\usepackage[utf8]{inputenc}
\usepackage[english]{babel}

\usepackage{braket}
\usepackage{amsmath}
\usepackage{amstext}
\usepackage{amssymb}
\usepackage{amsthm}
\usepackage{mathtools}
\usepackage{stmaryrd}
\usepackage{mathrsfs}
\usepackage{extarrows}
\usepackage{afterpage}
\usepackage{tabu}
\usepackage{array}

\usepackage{microtype}

\usepackage{graphicx}
\usepackage{amsfonts}

\usepackage{units}
\usepackage{enumerate}
\usepackage{url}
\usepackage{multirow}

\newcolumntype{C}[1]{>{\centering\let\newline\\\arraybackslash\hspace{0pt}}m{#1}}

\usepackage{IEEEtrantools}

\usepackage{tikz}
\usetikzlibrary{matrix,arrows,calc,fit,positioning}

\usepackage{hyperref}

\usepackage[margin=1cm]{caption}

\allowdisplaybreaks[4]

\theoremstyle{definition}
\newtheorem{theorem}{Theorem}[section]
\newtheorem{definition}[theorem]{Definition}
\newtheorem{remark}[theorem]{Remark}
\newtheorem{example}[theorem]{Example}
\newtheorem{lemma}[theorem]{Lemma}
\newtheorem{proposition}[theorem]{Proposition}
\newtheorem{corollary}[theorem]{Corollary}

\newtheorem{innercustomthm}{Theorem}
\newenvironment{customthm}[1]
  {\renewcommand\theinnercustomthm{#1}\innercustomthm}
  {\endinnercustomthm}

\newcommand{\N}{\mathbb{N}}
\newcommand{\Z}{\mathbb{Z}}

\newcommand{\D}{\mathcal{D}}
\renewcommand{\P}{\mathcal{P}}

\newcommand{\A}{\mathcal{A}}
\newcommand{\T}{\mathcal{T}}
\newcommand{\E}{\mathcal{E}}

\newcommand{\donotbreak}[1]{{}$\kern-2\mathsurround${}
  \binoppenalty10000 \relpenalty10000 #1{}$\kern-2\mathsurround${}}

\newcommand{\precdot}{\mathbin{\prec{\mkern-7mu\cdotp}}}
\newcommand{\succdot}{\mathbin{\cdotp{\mkern-5.5mu\succ}}}

\newcommand{\one}{\hat{1}}
\newcommand{\zero}{\hat{0}}

\DeclareMathOperator{\rk}{rk}

\tikzstyle{treestyle}=[
  baseline=(current bounding box.north),
  level/.style={sibling distance=15mm},
  every node/.style={inner sep=2pt},
  level distance=12mm
]
\tikzstyle{treestylewide}=[
  baseline=(current bounding box.north),
  level/.style={sibling distance=22mm},
  every node/.style={inner sep=2pt},
  level distance=12mm
]
\tikzstyle{bloomstyle}=[
  every node/.style={circle, draw, fill, inner sep=0pt, minimum size=0.25cm},
]

\title{Shellability of generalized Dowling posets}
\author{Giovanni Paolini}

\begin{document}

\begin{abstract}
  A generalization of Dowling lattices was recently introduced by Bibby and
  Gadish, in a work on orbit configuration spaces. The authors left open the question as to whether these posets are shellable. In this paper we prove EL-shellability and use it to determine the homotopy type.
  Our result generalizes shellability of Dowling lattices and of posets of layers of abelian arrangements defined by root systems.
  We also show that subposets corresponding to invariant subarrangements are not shellable in general.
  
  \bigskip
  \noindent\textbf{Keywords:} Shellability, configuration spaces, posets, Dowling lattices.
  
  \medskip
  \noindent\textbf{Mathematics Subject Classification (2010):} Primary 06A07; Secondary 05E18, 52B22, 52C35.
\end{abstract}

\maketitle

\section{Introduction}

In a recent contribution to the study of orbit configuration spaces \cite{bibby2018combinatorics}, Bibby and Gadish introduced a class of posets $\D_n(G,S)$ which they called \emph{$S$-Dowling posets}.
Here $n$ is a positive integer, $S$ is a finite set, and $G$ is a finite group acting on $S$.
These posets arise as posets of layers of arrangements $\A_n(G,X)$ of ``singular subspaces'' in $X^n$, where $X$ is a space with a $G$-action.
They generalize both Dowling lattices \cite{dowling1973class} (which are obtained for $|S|=1$) and posets of layers of linear, toric and elliptic arrangements defined by root systems of type $C$ \cite{bibby2017representation} (obtained if $G=\Z_2$ acts trivially on $S$, and $|S| \in \{1,2,4\}$).

Dowling lattices have long been known to be shellable \cite{gottlieb1998cohomology}, and a recent work of the author with Delucchi and Girard establishes shellability of posets of layers of arrangements defined by root systems \cite{delucchi2017shellability}.
A natural question posed in \cite{bibby2018combinatorics} is to prove shellability for the $S$-Dowling posets $\D_n(G,S)$.
In Section \ref{sec:dowling-shellability} we solve this conjecture in a positive way (Theorem \ref{thm:dowling-shellability}):

\begin{customthm}{A}
  The poset $\D_n(G,S) \cup \{\one\}$ is EL-shellable.
  \label{thm:A}
\end{customthm}

The order complex of a shellable poset is homotopy equivalent to a wedge of spheres.
In Section \ref{sec:homotopy-type} we determine the number of these spheres, by counting certain rooted trees (Theorem \ref{thm:homotopy-type}):

\begin{customthm}{B}
  The order complex of the poset $\D_n(G,S) \setminus \{\zero\}$ is homotopy equivalent to a wedge of
  \[ (-1)^\epsilon \, \prod_{i=0}^{n-1} (|S|-1 + |G|i) \]
  $(n-1-\epsilon)$-dimensional spheres, except for the empty poset $\bar\D_1(\{e\}, \varnothing)$.
  Here $\epsilon = 0$ for $S \neq \varnothing$ and $\epsilon = 1$ for $S = \varnothing$.
\end{customthm}

This refines the results of \cite{bibby2018combinatorics} about the characteristic polynomial and the homology of $\D_n(G,S)$.

In the study of the posets of layers of invariant arrangements, Bibby and Gadish also introduced subposets $\P_n(G,S,T) \subseteq \D_n(G,S)$ corresponding to any $G$-invariant subset $T\subseteq S$.
When $G=\Z_2$ acts trivially on $S$, suitable choices of $T$ yield posets of layers of arrangements defined by root systems of type $B$ and $D$ \cite{bibby2017representation}, which were proved to be shellable \cite{delucchi2017shellability}.
Therefore it is natural to ask if the subposets $\P_n(G,S,T)$ are shellable in general.
In Section \ref{sec:subposets} we exhibit a family of counterexamples, obtained when all the elements of $S$ have a trivial $G$-stabilizer.
However, we prove that $\P_n(G,S,T)$ is shellable if the $G$-action on $S\setminus T$ is trivial (Theorem \ref{thm:subposet-shellability}):

\begin{customthm}{C}
  If the $G$-action on $S\setminus T$ is trivial, the poset $\P_n(G,S,T) \cup \{\one\}$ is EL-shellable.
  \label{thm:C}
\end{customthm}

Finally, we determine the homotopy type of a larger class of subposets (Theorem \ref{thm:homotopy-type-subposets}):

\begin{customthm}{D}
  Let $n\geq 2$ and $|S| \geq 1$.
  Suppose that all the $G$-orbits in $S\setminus T$ either have cardinality 1, or have a trivial stabilizer.
  Then the poset $\P_n(G,S,T) \setminus \{\zero \}$ is homotopy equivalent to a wedge of $d$-dimensional spheres with
  \[ d =
    \begin{cases}
      n-1 & \text{if $T \neq \varnothing$ or at least one $G$-orbit is trivial} \\
      n-2 & \text{otherwise}.
    \end{cases}
  \]
\end{customthm}

\section{Preliminaries}

\subsection{Generalized Dowling posets}
\label{sec:generalized-dowling-posets}

The definition of the poset $(\D_n(G,S), \preceq)$ is as follows \cite[Section 2]{bibby2018combinatorics}.
Let $[n] = \{1,2,\dotsc, n\}$.
A \emph{partial $G$-partition} is a partition $\beta = \{B_1, \dotsc, B_\ell\}$ of the subset $\cup B_i \subseteq [n]$ together with functions $b_i \colon B_i \to G$ defined up to the following equivalence relation: $b_i \sim b_i'$ if $b_i = b_i'g$ for some $g\in G$.
The functions $b_i$ can be regarded as \emph{projectivized $G$-colorings}.
Define the \emph{zero block} of $\beta$ as $Z = [n] \setminus \cup B_i$.
Then $\D_n(G,S)$ is the set of partial $G$-partitions $\beta$ of $[n]$ together with an $S$-coloring of its zero block, i.e.\ a function $z\colon Z \to S$.

Following the conventions of \cite{bibby2018combinatorics}, we use an uppercase letter $B$ for a subset of $[n]$, the corresponding lowercase letter for the function $b\colon B \to G$, and $\tilde B$ for the data $(B, \bar b)$ where $\bar b$ is the equivalence class of $b\colon B \to G$.
Then elements of $\D_n(G,S)$ take the form $(\tilde\beta, z)$, where $\tilde\beta = \{ \tilde B_1, \dotsc, \tilde B_\ell\}$ and $z\colon Z \to S$ is the $S$-coloring of the zero block.

The set $\D_n(G,S)$ is partially ordered by covering relations (for which we use the symbol $\precdot$), given by either merging two blocks or coloring one by $S$:
\begin{enumerate}
  \item[(merge)] $(\tilde \beta \cup \{\tilde A, \tilde B\}, z) \precdot (\tilde \beta \cup \{\tilde C\}, z)$ where $C = A \cup B$ and $c=a\cup bg$ for some $g\in G$;
  \item[(color)] $(\tilde \beta \cup \{\tilde B\}, z) \precdot (\tilde \beta, z')$ where $z'$ is an extension of $z$ to $Z' = B \cup Z$ such that $z'|_B$ is a composition
  \[ B \xrightarrow{b} G \xrightarrow{f} S \]
  for some $G$-equivariant function $f$.
  Since $f$ is uniquely determined by $s = f(e) \in S$ (where $e$ is the identity element of $G$), we can equivalently say that $z'(i) = b(i) \cdot s$ for all $i \in B$.
\end{enumerate}
The poset $\D_n(G,S)$ is ranked by the rank function $\rk((\tilde\beta, z)) = n - |\tilde\beta|$.
For $S = \varnothing$, the zero blocks are always empty.
Therefore the rank of $\D_n(G,S)$ is $n-\epsilon$, where
\[ \epsilon =
  \begin{cases}
    0 & \text{if $S \neq \varnothing$} \\
    1 & \text{if $S = \varnothing$}.
  \end{cases}
\]

An element $(\tilde \beta, z) \in \D_n(G,S)$ will be written also as in the following example:
\[ [1_{g_1} 3_{g_3} \mid 2_{g_2} 4_{g_4} 6_{g_6} \parallel 5_{s_5} 7_{s_7}] \]
denotes the partial set partition $[13\mid 246]$ with projectivized $G$-colorings $[g_1 : g_3]$ and $[g_2 : g_4 : g_6]$, and zero block $\{5,7\}$ colored by $5 \mapsto s_5$ and $7 \mapsto s_7$.

Following \cite[Section 3.4]{bibby2018combinatorics}, we also introduce a subposet $\P_n(G,S,T) \subseteq \D_n(G,S)$ for any $G$-invariant subset $T \subseteq S$:
\[ \P_n(G,S,T) = \{ (\tilde \beta, z) \in \D_n(G,S) : |z^{-1}(O)| \neq 1 \text{ for every $G$-orbit $O \subseteq S\setminus T$} \}. \]
It arises as the poset of layers of a suitable invariant subarrangement $\A_n(G,X;T) \subseteq \A_n(G,X)$.
The subposet $\P_n(G,S,T)$ is ranked, with rank function induced by $\D_n(G,S)$.
The case $n=1$ does not yield anything new, because $\P_1(G,S,T) = \D_1(G,T)$.
For $n\geq 2$ we have $\rk \P_n(G,S,T) = \rk \D_n(G,S) = n-\epsilon$.

\subsection{EL-shellability}
\label{sec:shellability}

We refer to \cite{Bjorner1980,BW,BW1,BW2,wachs2006poset} for the definition and basic properties of shellability.
We are particularly interested in the notion of EL-shellability, which we now recall.

Let $P$ be a bounded poset.
Denote by $\one$ and $\zero$ the top and bottom elements of $P$, respectively.
Also let $\E(P) = \{ (x,y) \in P\times P \mid x \precdot y \}$ be the set of edges of the Hasse diagram of $P$ (i.e.\ the covering relations of $P$).

An \emph{edge labeling} of $P$ is a map $\lambda\colon \E(P) \to \Lambda$, where $\Lambda$ is some poset.
Given an edge labeling $\lambda$, each maximal chain $\gamma = (x \precdot p_1 \precdot \dotsb \precdot p_k \precdot y)$ between any two elements $x \preceq y$ has an associated word
\[ \lambda(\gamma) = \lambda(x, p_1) \, \lambda(p_1, p_2) \dotsm \lambda(p_k, y). \]
The chain $\gamma$ is said to be \emph{increasing} if the associated word $\lambda(\gamma)$ is strictly increasing, and \emph{decreasing} if the associated word is weakly decreasing.
Maximal chains in a fixed interval $[x,y]\subseteq P$ can be compared lexicographically, by using the lexicographic order on the corresponding words.

\begin{definition}[{\cite[Definitions 2.1 and 2.2]{Bjorner1980}, \cite[Definition 5.2]{BW1}}]
  Let $P$ be a bounded poset.
  An \emph{edge-lexicographical labeling} (or simply \emph{EL-labeling}) of $P$ is an edge labeling such that in each closed interval $[x,y] \subseteq P$ there is a unique increasing maximal chain, and this chain lexicographically precedes all other maximal chains of $[x,y]$.
  The poset $P$ is \emph{EL-shellable} if it admits an EL-labeling.
  \label{def:el-labeling}
\end{definition}

Suppose that $P$ is an EL-shellable bounded poset, and let $\bar P = P \setminus \{\zero, \one\}$.
Then the order complexes of $P$ and $\bar P$ are shellable (see \cite[Theorem 2.3]{Bjorner1980} and \cite[Theorem 5.8]{BW1}).
Also, the order complex of $\bar P$ is homotopy equivalent to a wedge of spheres indexed by the decreasing maximal chains of $P$ \cite[Theorem 5.9]{BW1}.
When $P$ is a ranked poset, all spheres have dimension $\rk(P)-2$ and their number equals $(-1)^{\rk(P)} \mu_P(\zero, \one)$ where $\mu_P$ is the M\"obius function of $P$ (this is a standard application of a theorem by Hall \cite[Theorem 3.8.6]{Stanley}, cf.\ \cite[Theorem 2]{delucchi2017shellability}).
Notice that $\mu_P(\zero, \one) = \chi_P(0) = - \chi_{P'}(1)$, where $P' = P \setminus \{\one\}$ and $\chi_Q$ is the characteristic polynomial of a poset $Q$.

\section{EL-shellability of \texorpdfstring{$S$}{S}-Dowling posets}
\label{sec:dowling-shellability}

The poset $\D_n(G,S)$ contains a bottom element $\zero$ but it is usually not bounded from above.
We therefore introduce the bounded poset $\hat\D_n(G,S) = \D_n(G,S) \cup \{\one\}$ with $x \prec \one$ for all $x \in \D_n(G,S)$.
In this section we are going to prove that $\hat \D_n(G,S)$ is EL-shellable for every positive integer $n$, finite set $S$ and finite group $G$ acting on $S$.
This solves \cite[Conjecture 2.7.1]{bibby2018combinatorics}.
The construction of the EL-labeling takes ideas from (and generalizes) the EL-labeling of \cite{delucchi2017shellability}.

\begin{definition}
  Consider an edge $(x,y) \in \E = \E(\hat\D_n(G,S))$ with $y \neq \one$.
  \begin{itemize}
    \item Suppose that $(x,y)$ is of type ``merge'', i.e.\ $x = (\tilde \beta \cup \{\tilde A, \tilde B\}, z)$ and $y = (\tilde \beta \cup \{\tilde C \}, z)$, where $C = A \cup B$ and $c = a \cup bg$.
    We say that $(x,y)$ is \emph{coherent} if $c(\min A) = c(\min B)$ and \emph{non-coherent} otherwise.
    If $(x,y)$ is non-coherent, define $\alpha(x,y) \in G \setminus \{e\}$ as
    \[ \alpha(x,y) =
      \begin{cases}
	c(\min B) \cdot c(\min A)^{-1} & \text{if $\min A < \min B$} \\
	c(\min A) \cdot c(\min B)^{-1} & \text{otherwise}.
      \end{cases}
    \]
    Notice that this definition only depends on the $\sim$\hspace{0.03cm}-equivalence class of $c$.
    The definition of $\alpha$ is motivated as follows: if $c$ is normalized by setting $c(\min C) = e$, then $\{ c(\min A), c(\min B) \} = \{ e, \alpha(x,y) \}$.
    
    \item Suppose that $(x,y)$ is of type ``color'', i.e.\ $x = (\tilde \beta \cup \{\tilde B\}, z)$ and $y = (\tilde \beta, z')$.
    Then we say that $(x,y)$ is \emph{colored}, and its color is $z'(\min B)$.
  \end{itemize}
  \label{def:edge-types}
\end{definition}

\begin{remark}
  The previous definition is a generalization of the one given in \cite{delucchi2017shellability}, with ``signed'' replaced by ``colored''.
\end{remark}

\begin{definition}[Edge labeling of $\hat\D_n(G,S)$]
  Fix any total order $s_1 < s_2 < \dotsb < s_m$ on $S$ and any total order on $G \setminus \{e\}$ (no compatibility with the group structure, nor with the action, is needed).
  Let $\lambda$ be the edge labeling of $\hat\D_n(G,S)$ defined as follows ($A$, $B$ and $C$ are as in Definition \ref{def:edge-types}).
  \[
    \lambda(x, y) =
    \begin{cases}
      (0, \,\max(\min A, \min B)) & \text{if $(x,y)$ is coherent} \\
      (2, \,\min C, \,\alpha(x,y)) & \text{if $(x,y)$ is non-coherent} \\
      (1, \, k) & \text{if $(x,y)$ is colored of color $s_k$} \\
      (1, \, 2) & \text{if $y = \one$.}
    \end{cases}
  \]
  The values of $\lambda$ are compared lexicographically.
  In the second case, $\lambda(x,y)$ is a triple and is lexicographically larger than all the pairs that occur in the other cases.
  \label{def:dowling-labeling}
\end{definition}

\begin{lemma}[Non-coherent increasing chains]
  Let $p_1\precdot p_2 \precdot \dotsb \precdot p_k$ be an increasing chain in $\D_n(G,S)$ such that $(p_i, p_{i+1})$ is non-coherent for all $i$.
  Let $\tilde A, \tilde B$ be (non-zero) blocks of $p_1$ such that $A \cup B \subseteq C$ for some (non-zero) block $\tilde C$ of $p_k$.
  Then $c(\min A) \neq c(\min B)$.
  \label{lemma:non-coherent-increasing-chains}
\end{lemma}

\begin{proof}
  Suppose that $\gamma = (p_1\precdot p_2 \precdot \dotsb \precdot p_k)$ is a chain of minimal length for which the lemma is false.
  By minimality, $p_k$ is the first element of $\gamma$ where $A$ and $B$ are contained in the same block.
  In other words, the edge $(p_{k-1},p_k)$ merges two blocks $A'\supseteq A$ and $B' \supseteq B$ of $p_{k-1}$ into the single block $C = A'\cup B'$ of $p_k$.
  Also by minimality, $p_1$ is the last element of $\gamma$ where both $A$ and $B$ are not contained in a larger block.
  Then assume without loss of generality that the edge $(p_1,p_2)$ merges $A$ and some other block of $p_1$ into a single block $A''\supseteq A$ of $p_2$.
  Therefore we have the inclusions $A \subset A'' \subseteq A' \subset C$.
  
  Let $\lambda(p_i, p_{i+1}) = (2, j_i, g_i)$. Then we have the following increasing sequence of labels:
  \[ (2, j_1, g_1) < (2, j_2, g_2) < \dotsb < (2, j_{k-1}, g_{k-1}). \]
  We have that $j_1 = \min A'' \geq \min C = j_{k-1}$, so $j_1 = j_2 = \dotsb = j_{k-1} = \min C$.
  This means that each edge of $\gamma$ consists of a merge which involves the element $\min C$.
  Also, $g_1 < g_2 < \dotsb < g_{k-1}$.
  If we normalize $c$ so that $c(\min C) = e$, by definition of $\lambda$ we have that:
  \begin{itemize}
    \item $c(\min A) =
      \begin{cases}
	e & \text{if $\min A = \min C$} \\
	g_1 & \text{otherwise};
      \end{cases}
    $
    \smallskip
    \item $c(\min B) = 
    \begin{cases}
	e & \text{if $\min B = \min C$ (this can only happen if $k=2$)} \\
	g_{k-1} & \text{otherwise.}
      \end{cases}
    $
  \end{itemize}
  If $k=2$, then $A'' = A \cup B$ and exactly one of $c(\min A)$ and $c(\min B)$ is equal to $e$.
  If $k > 2$, we have $c(\min B) = g_{k-1}$, which is different from both $e$ and $g_1$.
  In any case, $c(\min A) \neq c(\min B)$.
\end{proof}

\begin{theorem}[EL-shellability]
  The edge labeling $\lambda$ of Definition \ref{def:dowling-labeling} is an EL-labeling of $\hat\D_n(G,S)$.
  \label{thm:dowling-shellability}
\end{theorem}

\begin{proof}
  In order to check Definition \ref{def:el-labeling}, consider an interval $[x,y]$ of $\hat\D_n(G,S)$.
  For $x=y$ or $x \precdot y$ there is nothing to prove, so assume $\rk (y) - \rk (x) \geq 2$.
  Let $x = (\tilde \beta, z)$, with underlying partition $\tilde\beta = \{\tilde B_1, \dotsc, \tilde B_\ell\}$ and zero block $Z$.
  Order the blocks of $x$ so that $\min B_1 < \dotsb < \min B_\ell$.
  
  \smallskip
  \textbf{Case 1: $y = \one$.}
  Suppose that $\gamma=(x = p_0 \precdot p_1 \precdot \dotsb \precdot p_k \precdot \one)$ is an increasing maximal chain in the interval $[x,\one]$, with $k\geq 1$.
  Colored edges exist along $\gamma$ if and only if $S \neq \varnothing$ and $\tilde\beta \neq \varnothing$.
  Since $\gamma$ is increasing, we can deduce the following: the last edge $(p_k, \one)$ is labeled $(1,2)$; at most one edge (namely $(p_{k-1}, p_k)$) is labeled $(1,1)$ and is colored; all other edges are coherent and their labels are forced to be $(0,\min B_2), (0,\min B_3), \dotsc, (0,\min B_{\ell})$.
  Notice that $k=\ell$ if $(p_{k-1}, p_k)$ is colored (which happens if and only if $S\neq\varnothing$ and $\tilde\beta\neq\varnothing$), otherwise $k = \ell-1$.
  In any case $p_{\ell-1}$ has a single non-zero block $B = B_1 \cup \dotsb \cup B_\ell$, with $G$-coloring $b$ uniquely determined by $b(\min B_1) = \dotsb = b(\min B_\ell)$ since all edges from $x$ to $p_{\ell-1}$ are coherent.
  Therefore $p_{\ell-1}$ is uniquely determined by $x$.
  If $k=\ell$ the edge $(p_{k-1},p_k)$ is colored with color $s_1$, and thus $p_k = (\varnothing, z_k)$ is uniquely determined by the condition $z_k(\min B) = s_1$.
  Finally, there is exactly one coherent increasing chain from $x$ to $p_{\ell-1}$, in which $(p_i, p_{i+1})$ is the coherent edge which merges the blocks $B_1\cup \dotsb \cup B_{i+1}$ and $B_{i+2}$.
  
  The previous argument also shows how to construct an increasing maximal chain $\gamma=(x = p_0 \precdot p_1 \precdot \dotsb \precdot p_k \precdot \one)$ in $[x,\one]$, so there is exactly one such chain.
  For all $i \in \{0,\dotsc,k-1\}$, the label $\lambda(p_i, q)$ is minimized (only) when $q=p_{i+1}$.
  Therefore $\gamma$ is lexicographically minimal.
  
  \smallskip
  \textbf{Case 2: $y = (\tilde\beta', \varnothing)$, and $\tilde\beta'$ has only one non-singleton block $\tilde B$.}
  Let $b$ be the $G$-coloring of $\tilde B$ in $y$, normalized to have $b(\min B) = e$.
  Suppose without loss of generality that $B = B_1 \cup \dots \cup B_{\ell'}$, for some $\ell' \leq \ell$.
  Let $g_i = b(\min B_i)$ for $i=1,\dotsc, \ell'$.
  Notice that $g_1 = e$, since $\min B_1 = \min B$.
  
  Suppose that $\gamma=(x = p_0 \precdot p_1 \precdot \dotsb \precdot p_k = y)$ is an increasing maximal chain in $[x,y]$.
  Since the zero block of $y$ is empty, along $\gamma$ there are no colored edges.
  Then the edges are coherent from $x$ to $p_m$ for some $m \in \{0,\dotsc, k\}$, and non-coherent from $p_m$ to $y$.
  Let $p_m = (\{\tilde A_1, \dotsc, \tilde A_s\}, \varnothing)$.
  Since there is a coherent chain from $x$ to $p_m$, any block $A_j$ of $p_m$ is a union $B_{i_1} \cup \dotsb \cup B_{i_r}$ with $g_{i_1} = \dotsb = g_{i_r}$, and $b(\min A_j) = g_{i_1} = \dotsb = g_{i_r}$.
  Lemma \ref{lemma:non-coherent-increasing-chains} implies that $b(\min A_i) \neq b(\min A_j)$ for all $i\neq j$.
  Therefore $p_m$ is uniquely determined by $x$ and $y$.
  
  The part of $\gamma$ from $x$ to $p_m$ consists only of coherent edges, and it is uniquely determined by $x$ and $p_m$ as in Case 1.
  Consider now the part of $\gamma$ from $p_m$ to $y$.
  If $p_m = y$ there is nothing to prove, so suppose $p_m \neq y$.
  The labels from $p_m$ to $y$ take the form
  \[ (2, j_m, h_m) < (2, j_{m+1}, h_{m+1}) < \dotsb < (2, j_{k-1}, h_{k-1}). \]
  By definition of $\lambda$, we have $j_{k-1} = \min B$ and $j_m = \min (A_i \cup A_j)$ for some $i \neq j$. Since $A_i \cup A_j \subseteq B$, we can deduce that $j_m \geq j_{k-1}$.
  This implies $j_m = j_{m+1} = \dotsb = j_{k-1} = \min B$.
  Therefore every non-coherent edge consists of a merge which involves the element $\min B$.
  Also, we have that $h_m < h_{m+1} < \dotsb < h_{k-1}$.
  The elements $e, h_m, h_{m+1}, \dotsc, h_{k-1}$ of $G$ are all distinct, and by definition of $\lambda$ they coincide (up to some permutation) with $b(\min A_1), b(\min A_2), \dotsc, b(\min A_s)$.
  Then the chain from $p_m$ to $y$ is forced by the order $h_m < h_{m+1} < \dotsb < h_{k-1}$: first merge the block corresponding to $e$ with the block corresponding to $h_m$; then merge the resulting block with the block corresponding to $h_{m+1}$; and so on.
  At each step, the $G$-coloring is determined by $b$.
  
  We proved that an increasing chain $\gamma$ in $[x,y]$ is uniquely determined by $x$ and $y$, and our argument shows how to construct such a chain.
  We still need to prove that $\gamma$ is lexicographically minimal in $[x,y]$.
  Suppose that a lexicographically minimal chain $\gamma'$ first differs from $\gamma$ at some edge $(p_r, p_{r+1}')$, i.e.\ $\lambda(p_r, p_{r+1}') < \lambda(p_r, p_{r+1})$.
  \begin{itemize}
    \item If $r < m$, the edge $(p_r,p_{r+1})$ is coherent, so the edge $(p_r, p_{r+1}')$ must also be coherent.
    In order to remain in the interval $[x,y]$, a coherent merge between two blocks $C_1$ and $C_2$ of $p_r$ is possible only if $b(\min C_1) = b(\min C_2)$.
    Then $\lambda(p_r,p_{r+1}')$ is minimized for $p_{r+1}' = p_{r+1}$.
    
    \item If $r \geq m$, the chain $\gamma'$ coincides with $\gamma$ at least up to $p_m = (\{\tilde A_1, \dotsc, \tilde A_s\}, \varnothing)$. The edge $(p_r,p_{r+1}')$ cannot be coherent, because $b(\min A_i) \neq b(\min A_j)$ for all $i\neq j$.
    Then $(p_r, p_{r+1}')$ is non-coherent. The second entry of $\lambda(p_r,p_{r+1}')$ is at least $\min B$, so it must be equal to $\min B$ by minimality of $\gamma'$.
    This means that $(p_r, p_{r+1}')$ consists of a non-coherent merge which involves $\min B$.
    Then the possible values for the third entry of $\lambda(p_r,p_{r+1}')$ are $\{h_r, h_{r+1}, \dotsc, h_{k-1}\}$. The smallest one is $h_r$, which is attained for $p_{r+1}' = p_{r+1}$.
  \end{itemize}
  
  \smallskip
  \textbf{Case 3: $y = (\tilde\beta', z')$, and all blocks of $\tilde\beta'$ are singletons.}
  Suppose that $\gamma=(x = p_0 \precdot p_1 \precdot \dotsb \precdot p_k = y)$ is an increasing maximal chain in $[x,y]$.
  By the structure of $y$, the last edge $(p_{k-1},y)$ of $\gamma$ must be colored.
  Then the edges along $\gamma$ are coherent from $x$ to $p_m$ for some $m \in \{0,\dotsc, k\}$, and colored from $p_m$ to $y$.
  Let $p_m = (\{\tilde A_1, \dotsc, \tilde A_s\}, z)$.
  Since all the merges of $\gamma$ are coherent, if two blocks $B_i$ and $B_j$ of $x$ are contained in the same block $A$ of $p_m$, then we have $a(\min B_i) = a(\min B_j)$ and therefore $z'(\min B_i) = z'(\min B_j)$.
  In addition, $z'(\min A_i) \neq z'(\min A_j)$ for all distinct blocks $A_i, A_j$ of $p_m$, because the colored edges have strictly increasing (and thus distinct) colors.
  Putting everything together, two blocks $B_i$ and $B_j$ of $x$ are contained in the same block of $p_m$ if and only if $z'(\min B_i) = z'(\min B_j)$.
  This determines $p_m$ uniquely.
  
  The part of $\gamma$ from $x$ to $p_m$ is uniquely determine as in Case 1.
  Then there must be a colored edge for each (non-zero) block of $p_m$. Their colors are determined, so their order is also determined, because the sequence of colors must be increasing.
  
  Therefore the whole chain $\gamma$ is uniquely determined by $x$ and $y$, and once again the previous argument explicitly yields one such chain.
  Suppose that a lexicographically minimal chain $\gamma'$ in $[x,y]$ first differs from $\gamma$ at some edge $(p_r, p_{r+1}')$, i.e.\ $\lambda(p_r, p_{r+1}') < \lambda(p_r, p_{r+1})$.
  \begin{itemize}
    \item If $r < m$, the edge $(p_r,p_{r+1})$ is coherent, so the edge $(p_r, p_{r+1}')$ must also be coherent.
    In order to remain in the interval $[x,y]$, a coherent merge between two blocks $C_1$ and $C_2$ of $p_i$ is possible only if $z'(\min C_1) = z'(\min C_2)$.
    Then $\lambda(p_r,p_{r+1}')$ is minimized for $p_{r+1}' = p_{r+1}$.
    
    \item If $r \geq m$, the chain $\gamma'$ coincides with $\gamma$ at least up to $p_m = (\{\tilde A_1, \dotsc, \tilde A_s\}, z)$. The edge $(p_r,p_{r+1}')$ cannot be coherent, because $z'(\min A_i) \neq z'(\min A_j)$ for all $i\neq j$.
    Then $(p_r, p_{r+1}')$ is colored.
    The possible colors are given by the values of $z'$ on $\min A_1, \dotsc, \min A_s$.
    The smallest color still available is attained for $p_{r+1}' = p_{r+1}$.
  \end{itemize}
  
  \smallskip
  \textbf{Case 4: $y \neq \one$.}
  Let $y = (\tilde \beta', z')$ with underlying partition $\tilde\beta' = \{\tilde B_1', \dotsc, \tilde B_r'\}$ and zero block $Z'$.
  The interval $[x,y]$ is isomorphic to a product of intervals $[x_i, y_i] \subseteq \D_n(G,S)$ for $0 \leq i \leq r$, where: $y_0$ has the same zero block as $y$ (with the same $S$-coloring), and all other blocks are singletons; for $1 \leq i \leq r$, $y_i$ has exactly one non-singleton block which is equal to $\tilde B_i$, and an empty zero block.
  For each $i\neq j$, the sets of labels used in the intervals $[x_i,y_i]$ and $[x_j,y_j]$ are disjoint.
  By Case 2 and Case 3, $\lambda|_{[x_i,y_i]}$ is an EL-labeling for all $i$.
  Then also $\lambda|_{[x,y]}$ is an EL-labeling by \cite[Proposition 10.15]{BW2}.
\end{proof}

\begin{remark}
  The group structure of $G$ and the $G$-action on $S$ play a very little role in our EL-labeling.
  A similar unexpected separation between combinatorics and algebra was already observed in the characteristic polynomial of $\D_n(G,S)$ \cite[Remark 2.5.3]{bibby2018combinatorics}.
\end{remark}

\section{Homotopy type of \texorpdfstring{$S$}{S}-Dowling posets}
\label{sec:homotopy-type}

As in the introduction, let
\begin{equation}
\epsilon =
  \begin{cases}
    0 & \text{if $S \neq \varnothing$} \\
    1 & \text{if $S = \varnothing$}.
  \end{cases}
  \label{eq:epsilon}
\end{equation}
In this section we assume not to be in the degenerate case $S = \varnothing$, $G = \{e\}$ and $n=1$, because $\D_1(\{e\}, \varnothing) \setminus\{\zero\}$ is empty.

Since the poset $\hat\D_n(G,S)$ is (EL-)shellable and has rank $n+1-\epsilon$, the order complex of $\bar\D_n(G,S) = \D_n(G,S) \setminus \{\zero\}$ is homotopy equivalent to a wedge of $(n-1-\epsilon)$-dimensional spheres.
The homotopy type is therefore determined by the number of these spheres.
As recalled in Section \ref{sec:shellability}, we have (at least) two ways to determine this number: as an evaluation of the characteristic polynomial, and as the number of decreasing maximal chains in an EL-labeling.

The characteristic polynomial $\chi(t)$ of $\D_n(G,S)$ was computed in \cite[Theorem 2.5.2]{bibby2018combinatorics}:
\[
  \chi(t) =
  \begin{cases}
    \displaystyle\prod_{i=0}^{n-1} (t - |S| - |G|i) & \text{if $S\neq \varnothing$} \\[0.5cm]
    \displaystyle\prod_{i=1}^{n-1} (t - |G|i) & \text{if $S = \varnothing$}.
  \end{cases}
\]
\vskip0.2cm
\noindent Therefore the number of spheres is given by
\begin{IEEEeqnarray*}{rCl}
  (-1)^{\rk \D_n(G,S)} \chi(1) &=& (-1)^{n-\epsilon} \, \prod_{i=0}^{n-1} (1-|S|-|G|i) \\
  &=& (-1)^\epsilon \, \prod_{i=0}^{n-1} (|S|-1 + |G|i).
\end{IEEEeqnarray*}
Notice that the product vanishes for $|S|=1$.
This is correct, because in this case the poset $\bar\D_n(G,S)$ is bounded from above, and thus its order complex is contractible.

In the rest of this section we are going to prove the previous formula for the number of spheres by counting the decreasing maximal chains in $\hat\D_n(G,S)$, with respect to the EL-labeling of Definition \ref{def:dowling-labeling}.
Some of the arguments below are similar to those of \cite[Section 5]{delucchi2017shellability}.
We first recall the notion of \emph{increasing ordered tree} \cite{chen1994heap,gessel1995enumeration,klazar1997twelve,delucchi2017shellability}.

\begin{definition}
  An \emph{increasing ordered tree} is a rooted tree, with nodes in bijection with a finite subset $L \subset \N$ of labels, such that:
  \begin{itemize}
    \item each path from the root to any leaf has increasing labels (in particular, the root has label $\min L$);
    \item for each node, a total order of its children is specified.
  \end{itemize}
  If $L$ is not specified, it is assumed to be $\{0,1,\dotsc,n\}$ for some integer $n\geq 0$.
\end{definition}

\begin{figure}[htbp]
  \begin{tikzpicture}[treestyle]
  \node [circle,draw] (1) {0}
    child {node [circle,draw] (2) {1}
      child {node [circle,draw] (3) {2}
      }
    };
  \end{tikzpicture}
  \qquad
  \begin{tikzpicture}[treestyle]
  \node [circle,draw] (1) {0}
    child {node [circle,draw] (2) {1}
    }
    child {node [circle,draw] (3) {2}
  };
  \end{tikzpicture}
  \qquad
  \begin{tikzpicture}[treestyle]
  \node [circle,draw] (1) {0}
    child {node [circle,draw] (3) {2}
    }
    child {node [circle,draw] (2) {1}
  };
  \end{tikzpicture}
  \caption{All different increasing ordered trees on $3$ nodes. The second and the third tree differ in the total order of the children of the root.}
  \label{fig:trees}
\end{figure}

It is useful to introduce the following variant of increasing ordered trees, which generalizes the $q$-blooming trees of \cite[Definition 9]{delucchi2017shellability}.

\begin{definition}
  Let $q,r\geq 0$ be integers.
  A \emph{$(q,r)$-blooming tree} is an increasing ordered tree with $q$ extra indistinguishable unlabeled nodes appended to the root, and $r$ extra indistinguishable unlabeled nodes appended to each labeled node other than the root.
  The extra unlabeled nodes are called \emph{blooms}.
  The only thing that matters about blooms is their position in the total order of the children of a node.
  \label{def:blooming-tree}
\end{definition}

Blooms can be regarded as separators placed in the list of the children of a node.
See Figure \ref{fig:blooming-trees} for a few examples.

\begin{figure}
  \begin{tikzpicture}
    \begin{scope}[style=treestyle]
    \node [circle,draw] (0) {0}
      child {node [circle,draw] (1) {1}
	child {node [circle,draw] (2) {2}
	}
      };
    \end{scope}
    \begin{scope}[style=bloomstyle]
    \node [left = 0.5cm of 1] (a) {};
    \node [left = 0.5cm of a] (b) {};
    \node [left = 0.5cm of 2] (c) {};
    \node [below = 0.5cm of 2] (d) {};
    \draw (0) -- (a);
    \draw (0) -- (b);
    \draw (1) -- (c);
    \draw (2) -- (d);
    \end{scope}
    \node [right = 0.5cm of 2] (fake) {};
  \end{tikzpicture}
  \qquad
  \begin{tikzpicture}
    \begin{scope}[style=treestyle]
    \node [circle,draw] (0) {0}
      child {node [circle,draw] (1) {1}
	child {node [circle,draw] (2) {2}
	}
      };
    \end{scope}
    \begin{scope}[style=bloomstyle]
    \node [left = 0.5cm of 1] (a) {};
    \node [right= 0.5cm of 1] (b) {};
    \node [left = 0.5cm of 2] (c) {};
    \node [below = 0.5cm of 2] (d) {};
    \draw (0) -- (a);
    \draw (0) -- (b);
    \draw (1) -- (c);
    \draw (2) -- (d);
    \end{scope}
  \end{tikzpicture}
  \qquad
  \begin{tikzpicture}
    \begin{scope}[style=treestyle]
    \node [circle,draw] (0) {0}
      child {node [circle,draw] (1) {1}
	child {node [circle,draw] (2) {2}
	}
      };
    \end{scope}
    \begin{scope}[style=bloomstyle]
    \node [right = 0.5cm of 1] (a) {};
    \node [right = 0.5cm of a] (b) {};
    \node [left = 0.5cm of 2] (c) {};
    \node [below = 0.5cm of 2] (d) {};
    \draw (0) -- (a);
    \draw (0) -- (b);
    \draw (1) -- (c);
    \draw (2) -- (d);
    \end{scope}
  \end{tikzpicture}
  
  \vskip0.5cm
  
  \begin{tikzpicture}
    \begin{scope}[style=treestyle]
    \node [circle,draw] (0) {0}
      child {node [circle,draw] (1) {1}
	child {node [circle,draw] (2) {2}
	}
      };
    \end{scope}
    \begin{scope}[style=bloomstyle]
    \node [left = 0.5cm of 1] (a) {};
    \node [left = 0.5cm of a] (b) {};
    \node [right = 0.5cm of 2] (c) {};
    \node [below = 0.5cm of 2] (d) {};
    \draw (0) -- (a);
    \draw (0) -- (b);
    \draw (1) -- (c);
    \draw (2) -- (d);
    \end{scope}
  \end{tikzpicture}
  \qquad
  \begin{tikzpicture}
    \begin{scope}[style=treestyle]
    \node [circle,draw] (0) {0}
      child {node [circle,draw] (1) {1}
	child {node [circle,draw] (2) {2}
	}
      };
    \end{scope}
    \begin{scope}[style=bloomstyle]
    \node [left = 0.5cm of 1] (a) {};
    \node [right= 0.5cm of 1] (b) {};
    \node [right = 0.5cm of 2] (c) {};
    \node [below = 0.5cm of 2] (d) {};
    \draw (0) -- (a);
    \draw (0) -- (b);
    \draw (1) -- (c);
    \draw (2) -- (d);
    \end{scope}
  \end{tikzpicture}
  \qquad
  \begin{tikzpicture}
    \begin{scope}[style=treestyle]
    \node [circle,draw] (0) {0}
      child {node [circle,draw] (1) {1}
	child {node [circle,draw] (2) {2}
	}
      };
    \end{scope}
    \begin{scope}[style=bloomstyle]
    \node [right = 0.5cm of 1] (a) {};
    \node [right = 0.5cm of a] (b) {};
    \node [right = 0.5cm of 2] (c) {};
    \node [below = 0.5cm of 2] (d) {};
    \draw (0) -- (a);
    \draw (0) -- (b);
    \draw (1) -- (c);
    \draw (2) -- (d);
    \end{scope}
    \node [left = 0.5cm of 2] (fake) {};
  \end{tikzpicture}
  
  \caption{All different $(2,1)$-blooming trees constructed from the leftmost increasing ordered tree of Figure \ref{fig:trees}.
  These are $6$ of the $18$ different $(2,1)$-blooming trees on $3$ (labeled) nodes.
  Blooms are shown as smaller black (unlabeled) nodes.
  }
  \label{fig:blooming-trees}
\end{figure}

\begin{lemma}
  The number of $(q,r)$-blooming trees on $n+1$ (labeled) nodes is
  \[ \prod_{i=0}^{n-1} (q+1+(r+2)i). \]
  \label{lemma:number-trees}
\end{lemma}

\begin{proof}
  The proof is by induction on $n$.
  For $n=0$ there is only one $(q,r)$-blooming tree, consisting of the root with $q$ blooms attached to it.
  Let $T$ be any $(q,r)$-blooming tree on $n$ nodes with $n\geq 1$.
  The tree $T$ has $q+(n-1)r$ blooms, so there are exactly $2n-1+q+(n-1)r$ positions where an additional node with label $n$ can be attached (together with its $r$ new blooms) in order to obtain a $(q,r)$-blooming tree on $n+1$ nodes.
  Every $(q,r)$-blooming tree on $n+1$ nodes is obtained exactly once in this way.
\end{proof}

For $q=r=0$ one obtains the classical result \cite{chen1994heap, gessel1995enumeration, klazar1997twelve} which states that the number of increasing ordered trees on $n+1$ nodes is $(2n-1)!!$.

\begin{theorem}[Homotopy type of $\bar \D_n(G,S)$]
  The order complex of the poset $\bar\D_n(G,S)$ is homotopy equivalent to a wedge of
  \[ (-1)^\epsilon \, \prod_{i=0}^{n-1} (|S|-1 + |G|i) \]
  $(n-1-\epsilon)$-dimensional spheres (where $\epsilon$ is defined in eq.\ \eqref{eq:epsilon}), except for the empty poset $\bar\D_1(\{e\}, \varnothing)$.
  \label{thm:homotopy-type}
\end{theorem}

\begin{proof}
  We want to compute the cardinality of the set $D$ of the decreasing maximal chains from $\zero$ to $\one$ in $\hat\D_n(G,S)$.
  A chain $\gamma \in D$ consists of:
  \begin{itemize}
    \item a sequence of non-coherent edges, labeled $(2,*,*)$;
    \item then, a sequence of colored edges, labeled $(1,*)$ with $* \geq 2$;
    \item finally, one edge labeled $(1,2)$.
  \end{itemize}
  
  \smallskip
  \textbf{Case 1: $|G| = 1$.}
  In this case, non-coherent edges do not exist.
  Then a decreasing chain $\gamma \in D$ is determined by a permutation of $[n]$ (which encodes the order in which the elements of $[n]$ are colored) and a decreasing sequence of $n$ labels
  \[ (1,|S|) \geq (1,h_1) \geq (1,h_2) \geq \dotsb \geq (1,h_n) \geq (1,2) \]
  (which are to be assigned to the colored edges of $\gamma$).
  The number of decreasing chains in $\D_n(\{e\}, S)$ is therefore
  \[ n! \cdot \binom{n+|S|-2}{n} = \frac{(n+|S|-2)!}{(|S|-2)!} \]
  if $|S| \geq 2$, and $0$ if $|S| \leq 1$.
  From now on, assume $|G| \geq 2$.
    
  \smallskip
  \textbf{Case 2: $|S| = 1$.}
  Every chain contains at least one colored edge which is labeled $(1,1)$.
  Thus there are no decreasing chains.
  
  \smallskip
  \textbf{Case 3: $|G| \geq 2$ and $|S| \geq 2$.}
  We are going to construct a bijection $\psi$ between $D$ and the set $\T$ of $(|S|-2,\, |G|-2)$-blooming trees on $n+1$ nodes.
  Let $S = \{ s_1 < s_2 < \dotsb < s_m \}$ and $G\setminus \{e\} = \{ g_1 < g_2 < \dotsb < g_k\}$.
  The idea is: for every non-coherent merge of blocks $A$ and $B$, there is an edge between $\min A$ and $\min B$; for every $S$-coloring of a block $B$, there is an edge between $0$ and $\min B$; the blooms indicate when to stop using a certain $s_i$ (or $g_i$) and start using $s_{i-1}$ (or $g_{i-1}$).
  
  The bijection $\psi\colon D\to \T$ is defined as follows.
  Let $\gamma \in D$ be a decreasing chain.
  In order to construct the tree $\psi(\gamma) \in \T$, start with a disconnected graph on $n+1$ vertices labeled $0,1,\dotsc,n$.
  Every time we say ``attach the node $u$ to the node $v$'' we mean ``create an edge between $v$ and $u$, so that $u$ becomes the \emph{last} child of $v$ in the total order of the children of $v$.''
  
  First examine the colored edges along $\gamma$, in order.
  For each colored edge $(x,y)$, which colors a block $B$ and has label $\lambda(x,y) = (1,i)$, do the following.
  \begin{enumerate}
    \item If the root (i.e.\ the node $0$) has less than $m-i$ blooms, attach a new bloom to it; repeat this step until the root has exactly $m-i$ blooms.
    \item Attach the node $\min B$ to the root.
  \end{enumerate}
  Notice that, by monotonicity of the labels along $\gamma$, the number of blooms required in step (1) is weakly increasing and it varies between $0$ and $m-2$.

  Then examine the non-coherent edges along $\gamma$, in order.
  For each non-coherent edge $(x,y)$, which merges blocks $A$ and $B$ with $\min A < \min B$ and has label $\lambda(x,y) = (2,\min A, g_i)$, do the following.
  \begin{enumerate}
    \item If the node $\min A$ has less than $k-i$ blooms, attach a new bloom to it; repeat this step until the node $\min A$ has exactly $k-i$ blooms.
    \item Attach the node $\min B$ to the node $\min A$.
  \end{enumerate}
  Again, by monotonicity of the labels along $\gamma$, the number of blooms required in step (1) is weakly increasing and it varies between $0$ and $k-1 = |G|-2$.
  
  Attach new blooms to the tree, so that the root has $m-2$ blooms and every other labeled node has $k-1$ blooms.
  Define $\psi(\gamma)$ as the tree resulting from this construction.
  See Figure \ref{fig:bijection-example} for an example.
  
  \begin{figure}
    \begin{tikzpicture}[baseline]
      \node (5) {$\one$};
      \node [below = 0.3cm of 5] (4) {$[\varnothing \parallel 1_{s_2} 2_{g_2\cdot s_2} 3_{s_3} 4_{g_1 \cdot s_2}]$};
      \node [below = 0.3cm of 4] (3) {$[1_e 2_{g_2} 4_{g_1} \parallel 3_{s_3}]$};
      \node [below = 0.3cm of 3] (2) {$[1_e 2_{g_2} 4_{g_1} \mid 3_e \parallel \varnothing]$};
      \node [below = 0.3cm of 2] (1) {$[1_e 2_{g_2} \mid 3_e \mid 4_e \parallel \varnothing]$};
      \node [below = 0.3cm of 1] (0) {$[1_e \mid 2_e \mid 3_e \mid 4_e \parallel \varnothing]$};
      \draw (0) -- (1);
      \draw (1) -- (2);
      \draw (2) -- (3);
      \draw (3) -- (4);
      \draw (4) -- (5);
    \end{tikzpicture}
    \qquad
    \begin{tikzpicture}[baseline=0.5cm]
      \begin{scope}[style=treestylewide]
      \node [circle,draw] (0) {0}
	child {node [circle,draw] (3) {3}
	}
	child {node [circle,draw] (1) {1}
	  child {node [circle,draw] (2) {2}
	  }
	  child {node [circle,draw] (4) {4}
	  }
      };
      \end{scope}
      \begin{scope}[style=bloomstyle]
      \node [left = 0.5cm of 3] (a) {};
      \node [left = 0.5cm of a] (b) {};
      \node [below = 0.8cm of 0] (c) {};
      \node [below = 0.5cm of 3] (d) {};
      \node [below = 0.8cm of 1] (e) {};
      \node [below = 0.5cm of 2] (f) {};
      \node [below = 0.5cm of 4] (g) {};
      \draw (0) -- (a);
      \draw (0) -- (b);
      \draw (0) -- (c);
      \draw (3) -- (d);
      \draw (1) -- (e);
      \draw (2) -- (f);
      \draw (4) -- (g);      
      \end{scope}
      \node [right = 0.5cm of 2] (fake) {};
    \end{tikzpicture}
    
    \caption{A decreasing chain (for $n=4$, $|G|=3$, and $|S|=5$), and the corresponding $(3,1)$-blooming tree.}
    \label{fig:bijection-example}
  \end{figure}
  
  To prove that $\psi$ is a bijection, we explicitly define its inverse $\psi^{-1}\colon \T \to D$.
  Let $T \in \T$ be a tree.
  Start with $\gamma = (\zero)$ (a chain with one element).
  Consider the set of couples
  \[ \{ (u,v) \mid \text{$u,v$ are labeled nodes of $T$, and $v$ is a child of $u$} \}, \]
  totally ordered by: $(u_1,v_1) < (u_2,v_2)$ if $u_1 > u_2$, or $u_1=u_2$ and $v_1$ comes before $v_2$ in the total order of the children of $u_1=u_2$.
  Each of the nodes $1,\dotsc, n$ appears exactly once as the second entry of a couple $(u,v)$.
  With $(u,v)$ running through this ordered set of couples, do the following.
  \begin{itemize}
    \item Let $i$ be the number of blooms attached to $u$ which come before $v$ in the total order of the children of $u$.
    
    \item If $x \in \D_n(G,S)$ is the last element of $\gamma$, construct $y \succdot x$ as follows.
    \begin{itemize}
      \item \emph{Case $u>0$.}
      Merge the block containing $u$ and the block containing $v$ so that $\lambda(x,y) = (2,u,g_{k-i})$.
      
      \item \emph{Case $u=0$.}
      Color the block containing $v$ so that $\lambda(x,y) = (1,m-i)$.
    \end{itemize}
    In both cases, the element $y \succdot x$ is uniquely determined by the given conditions.
    
    \item Extend $\gamma$ by adding $y$ after $x$.
  \end{itemize}
  Extend $\gamma$ once more, by adding $\one$ as the last element.
  Then $\psi^{-1}(T) = \gamma$.
  
  By Lemma \ref{lemma:number-trees}, the number of $(|S|-2, \, |G|-2)$-blooming trees on $n+1$ nodes is
  \[ \prod_{i=0}^{n-1} (|S|-1 + |G|i). \]
  This is also the cardinality of $D$.
  
  \smallskip
  \textbf{Case 4: $S = \varnothing$.}
  Differently from before, in this case there are no colored edges, and the zero blocks are always empty.
  Then $D$ is in bijection with the set of $(|G|-2, \, |G|-2)$-blooming trees on $n$ vertices labeled $1,2,\dotsc, n$.
  The bijection is constructed as in the case $|S| \geq 2$, except that there is no special ``node $0$'' anymore.
  By Lemma \ref{lemma:number-trees}, the number of such trees is
  \[ \prod_{i=0}^{n-2} (|G|-1 + |G|i) = \prod_{i=1}^{n-1} (-1 + |G|i) = - \prod_{i=0}^{n-1} (-1 + |G|i). \]
  This is also the cardinality of $D$.
\end{proof}

In the case of Dowling lattices $\D_n(G)$, obtained by setting $|S|=1$, Theorem \ref{thm:homotopy-type} is trivial because $\D_n(G)$ contains a top element $(\varnothing, \hat z)$.
The subposet $\tilde\D_n (G) = \D_n(G) \setminus \{\zero, (\varnothing, \hat z)\}$ is shellable, and its order complex is homotopy equivalent to a wedge of spheres in bijection with the decreasing chains from $\zero$ to $(\varnothing, \hat z)$.
By an argument similar to that of Theorem \ref{thm:homotopy-type}, these decreasing chains are counted by $(0, \, |G|-2)$-blooming trees on $n+1$ vertices.
By Lemma \ref{lemma:number-trees}, their number is
\[ \prod_{i=0}^{n-1} (1 + |G|i). \]
A similar description of the generators of the homology of Dowling lattices was given in \cite[Section 4]{gottlieb2000cohomology}, in terms of labeled forests.
The formula for the number of generators was first found in \cite{dowling1973class}.

\section{Shellability of subposets}
\label{sec:subposets}

We now turn our attention to the subposet $\P_n(G,S,T)$ of $\D_n(G,S)$ introduced in Section \ref{sec:generalized-dowling-posets}, where $T$ is a $G$-invariant subset of $S$.
Throughout this section, assume $n\geq 2$ (because $\P_1(G,S,T) = \D_1(G,T)$) and $|S| \geq 1$.
Then $\rk \P_n(G,S,T) = n$.
Also, let $\bar \P_n(G,S,T) = \P_n(G,S,T) \setminus \{ \zero \}$.

We are going to prove that $\P_n(G,S,T)$ is shellable if $G$ acts trivially on $S \setminus T$.
In general, however, $\P_n(G,S,T)$ is not shellable.
In the first part of this section we construct a wide family of non-shellable examples (Proposition \ref{prop:trivial-stabilizers}), whereas in the second part we prove shellability if the $G$-action on $S\setminus T$ is trivial (Theorem \ref{thm:subposet-shellability}).
Let us start with two simple examples.

\begin{example}
  Let $G = \Z_2 = \{e, g\}$ act non-trivially on $S = \{+, -\}$, as in \cite[Example 2.2.3]{bibby2018combinatorics}.
  The Hasse diagram of $\P_2(G, S, \varnothing)$ is shown in Figure \ref{fig:subposet-Z2}.
  The order complex of $\bar \P_2(G,S,\varnothing)$ is $1$-dimensional and disconnected, therefore it is not shellable.
  This example is a special case of Proposition \ref{prop:trivial-stabilizers} below.
  In view of Theorem \ref{thm:subposet-shellability}, this is the smallest non-shellable example.
  \label{example:subposet-Z2}
\end{example}

\begin{figure}
  \centering
  \begin{tikzpicture}
    \node (0) {$[1_e \mid 2_e \parallel \varnothing]$};
    \node [above = 1cm of 0] (1) {};
    \node [left = 0.6cm of 1] (1c) {$[1_e 2_e \parallel \varnothing]$};
    \node [right = 0.6cm of 1] (1d) {$[1_e 2_g \parallel \varnothing]$};
    \node [above = 1cm of 1] (2) {};
    \node [left = 0.2cm of 2] (2b) {$[\varnothing \parallel 1_- 2_-]$};
    \node [left = 0.5cm of 2b] (2a) {$[\varnothing \parallel 1_+ 2_+]$};
    \node [right = 0.2cm of 2] (2c) {$[\varnothing \parallel 1_+ 2_-]$};
    \node [right = 0.5cm of 2c] (2d) {$[\varnothing \parallel 1_- 2_+]$};    
    \draw (0.north) -- (1c.south);
    \draw (0.north) -- (1d.south);
    \draw (1c.north) -- (2a.south);
    \draw (1c.north) -- (2b.south);
    \draw (1d.north) -- (2c.south);
    \draw (1d.north) -- (2d.south);
  \end{tikzpicture}

  \caption{The Hasse diagram of $\P_2(\Z_2, \{+,-\}, \varnothing)$ where $\Z_2$ acts non-trivially on $\{+,-\}$ (see Example \ref{example:subposet-Z2}).}
  \label{fig:subposet-Z2}
\end{figure}

\begin{example}
  Let $G = \Z_4 = \{e, g, g^2, g^3\}$ act non-trivially on $S = \{+, -\}$.
  The Hasse diagram of $\P_2(G, S, \varnothing)$ is shown in Figure \ref{fig:subposet-Z4}.
  As in Example \ref{example:subposet-Z2}, the order complex $\Delta$ of $\bar \P_2(G,S,\varnothing)$ is $1$-dimensional and disconnected, therefore it is not shellable.
  In this example, $\Delta \simeq S^1 \sqcup S^1$ is not even a wedge of spheres.
  \label{example:subposet-Z4}
\end{example}

\begin{figure}
  \centering
  \begin{tikzpicture}
    \node (0) {$[1_e \mid 2_e \parallel \varnothing]$};
    \node [above = 1cm of 0] (1) {};
    \node [right = 0.2cm of 1] (1b) {$[1_e 2_g \parallel \varnothing]$};
    \node [left = 0.2cm of 1] (1c) {$[1_e 2_{g^2} \parallel \varnothing]$};
    \node [left = 0.5cm of 1c] (1a) {$[1_e 2_e \parallel \varnothing]$};
    \node [right = 0.5cm of 1b] (1d) {$[1_e 2_{g^3} \parallel \varnothing]$};
    \node [above = 1cm of 1c] (2b) {$[\varnothing \parallel 1_- 2_-]$};
    \node [above = 1cm of 1a] (2a) {$[\varnothing \parallel 1_+ 2_+]$};
    \node [above = 1cm of 1b] (2c) {$[\varnothing \parallel 1_+ 2_-]$};
    \node [above = 1cm of 1d] (2d) {$[\varnothing \parallel 1_- 2_+]$};    
    \draw (0.north) -- (1a.south);
    \draw (0.north) -- (1b.south);
    \draw (0.north) -- (1c.south);
    \draw (0.north) -- (1d.south);
    \draw (1a.north) -- (2a.south);
    \draw (1a.north) -- (2b.south);
    \draw (1b.north) -- (2c.south);
    \draw (1b.north) -- (2d.south);
    \draw (1c.north) -- (2a.south);
    \draw (1c.north) -- (2b.south);
    \draw (1d.north) -- (2c.south);
    \draw (1d.north) -- (2d.south);
  \end{tikzpicture}
  
  \caption{The Hasse diagram of $\P_2(\Z_4, \{+,-\}, \varnothing)$ where $\Z_4$ acts non-trivially on $\{+,-\}$ (see Example \ref{example:subposet-Z4}).}
  \label{fig:subposet-Z4}
\end{figure}

We are going to prove a ``reduction lemma'' that allows to construct a wide family of non-shellable examples.
At the same time, it gives interesting homotopy equivalences between subposets of different $S$-Dowling posets.

\begin{lemma}[Orbit reduction]
  Suppose that $O \subseteq S\setminus T$ is a $G$-orbit with a trivial stabilizer (i.e.\ $|O| = |G|$).
  Then there is a homotopy equivalence $\bar \P_n(G,S,T) \simeq \bar \P_n(G, S\setminus O, T)$.
  \label{lemma:reduction-lemma}
\end{lemma}

\begin{proof}
  Fix an element $\bar s \in O$.
  Since $O$ has a trivial stabilizer, this choice induces a bijection $\smash{\varphi \colon G \xrightarrow{\cong} O}$ given by $g \mapsto g\cdot \bar s$.
  We are going to construct a (descending) closure operator $f \colon \P_n(G,S,T) \to \P_n(G, S\setminus O, T) \subseteq \P_n(G,S,T)$, i.e.\ an order-preserving map satisfying $f(x) \leq x$ for all $x \in \P_n(G,S,T)$.
  
  Let $x = (\tilde \beta, z) \in \P_n(G,S,T)$, and let $Z$ be the zero block of $x$.
  Consider the subset $Z_O = \{ i \in Z \mid z(i) \in O \}$ of $Z$ consisting of the elements colored by the orbit $O$.
  To define $f(x)$, remove $Z_O$ from the zero block and add $B = Z_O$ as a non-zero block with $G$-coloring $b = \varphi^{-1} \circ z|_{Z_O}$:
  \[ f(x) = (\tilde \beta \cup \{ \tilde B \}, \, z|_{Z \setminus Z_O}). \]
  Notice that $f$ does not depend on the initial choice of $\bar s \in O$.
  Indeed, different choices of $\bar s$ yield $\sim$\,-equivalent $G$-colorings of $B$.
  
  Clearly $f(x) \leq x$, because either $f(x) = x$ (if $Z_O = \varnothing$) or $x$ can be obtained from $f(x)$ by coloring the block $B$.
  Also, $f$ is order-preserving:
  \begin{itemize}
    \item if $(x,y) \in \E(\P_n(G,S,T))$ is an edge of type ``merge'', or is an edge of type ``color'' which uses colors in $S \setminus O$, then $(f(x), f(y))$ is an edge of the same type;
    \item if $(x,y) \in \E(\P_n(G,S,T))$ is an edge of type ``color'' which uses colors in $O$, then either $x=f(x)=f(y)$ (if $Z_O = \varnothing$ in $x$) or $(f(x), f(y))$ is an edge of type ``merge''.
  \end{itemize}
  In addition, $f$ is the identity on $\P_n(G, S\setminus O, T)$, and so $f$ is surjective.
  
  In the definition of $f(x)$, either $f(x) = x$ (if $Z_O = \varnothing$) or $f(x)$ contains the block $B$ which has at least $2$ elements (because $O \subseteq S\setminus T$; see the definition of the subposet $\P_n(G,S,T)$).
  Therefore $f^{-1}(\zero) = \{\zero\}$.
  Then $f$ restricts to a surjective closure operator $\bar f \colon \bar \P_n(G,S,T) \to \bar \P_n(G,S\setminus O, T)$.
  By \cite[Corollary 10.12]{bjorner1995topological}, $\bar f$ induces a homotopy equivalence between the associated order complexes.
\end{proof}

\begin{remark}
  The closure operator $f$ of the previous proof also satisfies $f^2 = f$.
  Then, by \cite[Theorem 2.1]{kozlov2006simple}, we actually have that the order complex of $\bar \P_n(G,S,T)$ collapses onto the order complex of $\bar \P_n(G, S\setminus O, T)$.
\end{remark}

\begin{remark}
  The combinatorics of $\bar\P_n(G,S,T)$ is related to the topology of the complement of  an arrangement $\A_n(G,X)$ of singular subspaces in $X^n$.
  In this sense an orbit $O$ as in the statement of Lemma \ref{lemma:reduction-lemma} is ``redundant'', as it consists of non-singular points.
  This provides a topological interpretation of Lemma \ref{lemma:reduction-lemma}.
\end{remark}

\begin{proposition}
  Let $n \geq 2$ and $|S| \geq 1$.
  Suppose that all the $G$-orbits in $S\setminus T$ have a trivial stabilizer.
  Then $\bar \P_n(G,S,T)$ is homotopy equivalent to a wedge of $d$-dimensional spheres with
  \[ d =
    \begin{cases}
      n-1 & \text{if $T \neq \varnothing$} \\
      n-2 & \text{if $T = \varnothing$}.
    \end{cases}
  \]
  In particular, if $|G| \geq 2$ and all the $G$-orbits in $S$ have a trivial stabilizer, the poset $\P_n(G,S,\varnothing)$ is not shellable.
  \label{prop:trivial-stabilizers}
\end{proposition}

\begin{proof}
  For the first part, a repeated application of Lemma \ref{lemma:reduction-lemma} yields $\bar \P_n(G,S,T) \simeq \bar \P_n(G,T,T) = \bar \D_n(G,T)$. 
  Then the homotopy type of $\bar D_n(G,T)$ is given by Theorem \ref{thm:homotopy-type}.
  In particular, the dimension $d$ of the spheres equals $\rk D_n(G,T) - 1$, which is $n-1$ for $T \neq \varnothing$ and $n-2$ for $T = \varnothing$.
  
  For the second part, $\bar \P_n(G,S,\varnothing) \simeq \bar \D_n(G, \varnothing)$.
  Since $|S| \geq 1$, the poset $\P_n(G,S,\varnothing)$ has rank $n$.
  For $\P_n(G,S,\varnothing)$ to be shellable, the order complex of $\bar \P_n(G,S,\varnothing)$ must be homotopy equivalent to a wedge of $(n-1)$-dimensional spheres.
  The hypothesis $|G| \geq 2$ ensures that $\bar \P_n(G,S,\varnothing) \simeq \bar \D_n(G,\varnothing)$ is a wedge of a positive number of $(n-2)$-dimensional spheres (by Theorem \ref{thm:homotopy-type}).
  Then $\P_n(G,S,\varnothing)$ is not shellable.
\end{proof}

The second part of this proposition yields a large family of examples of non-shellable subposets of $S$-Dowling posets, generalizing Example \ref{example:subposet-Z2}.
At the same time it shows that this family is still well-behaved, as $\bar \P_n(G,S,\varnothing)$ is homotopy equivalent to a wedge of spheres.

We now prove that $\P_n(G,S,T)$ is shellable if the $G$-action on $S\setminus T$ is trivial.
This generalizes \cite[Theorem 6]{delucchi2017shellability}.
Let $\hat \P_n(G,S,T) = \P_n(G,S,T) \cup \{\one\} \subseteq \hat\D_n(G,S)$.
For an element $x = (\tilde \beta, z) \in \P_n(G,S,T)$, define $S(x) \subseteq S$ as the image of the coloring map $z\colon Z \to S$.
Consider the following edge labeling, which is a slightly modified version of the edge labeling of Definition \ref{def:dowling-labeling} (cf.\ \cite[Definition 6]{delucchi2017shellability}).

\begin{definition}[Edge labeling of $\hat\P_n(G,S,T)$]
  Fix arbitrary total orders on $S$ and on $G \setminus \{e\}$.
  For a subset $R\subseteq S$, let $R_{\leq s} = \{ r \in R \mid r \leq s\}$.
  Let $\mu$ be the edge labeling of $\hat\P_n(G,S,T)$ defined as follows ($A$, $B$ and $C$ are as in Definition \ref{def:edge-types}).
  \[
    \mu(x, y) =
    \begin{cases}
      (0, \,\max(\min A, \min B)) & \text{if $(x,y)$ is coherent} \\
      (2, \,\min C, \,\alpha(x,y)) & \text{if $(x,y)$ is non-coherent} \\
      (1, \, |S(x)_{\leq s}|) & \text{if $(x,y)$ is colored of a color $s \in S(x)$} \\
      (1, \, |S_{\leq s} \cup S(x)|) & \text{if $(x,y)$ is colored of a color $s \not\in S(x)$} \\
      (1, \, 2) & \text{if $y = \one$.}
    \end{cases}
  \]
  \label{def:subposet-labeling}
\end{definition}

The difference with the edge labeling $\lambda$ of Definition \ref{def:dowling-labeling} is in the labels of colored edges: $\lambda$ only depends on the color $s$, whereas $\mu$ favors colors which already belong to $S(x)$.

\begin{theorem}[EL-shellability of subposets]
  Let $n\geq 2$, and suppose that the $G$-action on $S\setminus T$ is trivial.
  Then the edge labeling $\mu$ of Definition \ref{def:subposet-labeling} is an EL-labeling of $\hat \P_n(G,S,T)$.
  \label{thm:subposet-shellability}
\end{theorem}

\begin{proof}
  Since the edge labelings $\lambda$ and $\mu$ almost coincide, most of the proof of Theorem \ref{thm:dowling-shellability} also applies here.
  We refer to that proof (with its notations), and we only highlight the differences.
  Let $[x,y]$ be an interval in $\hat \P_n(G,S,T)$.
  
  \smallskip
  \textbf{Case 1: $y = \one$.}
  The only difference is that, if $k=\ell$ and $Z\neq \varnothing$, the edge $(p_{k-1},p_k)$ is colored with color $\min S(x)$ (and not with color $s_1$).
  This modification assures that $p_k$ belongs to the subposet $\P_n(G,S,T)$: if $|Z| \neq \varnothing$, the color of the edge $(p_{k-1}, p_k)$ was already used in $x$; if $|Z| = \varnothing$, the edge $(p_{k-1},p_k)$ colors $n\geq 2$ elements at the same time.
  
  \smallskip
  \textbf{Case 2: $y = (\tilde \beta', \varnothing)$, and $\tilde\beta'$ has only one non-singleton block $\tilde B$.}
  In this case the edge labelings $\lambda$ and $\mu$ coincide, and $[x,y]$ is also an interval of $\D_n(G,S)$.
  Therefore the proof works without changes.
  
  \smallskip
  \textbf{Case 3: $y = (\tilde\beta', z')$, and all blocks of $\tilde\beta'$ are singletons.}
  Here we only have to show that the increasing chain $\gamma = (x=p_0 \precdot p_1 \precdot \dotsb \precdot p_k = y)$ is contained in the subposet $\P_n(G,S,T)$.
  Until the element $p_m \in \gamma$, only coherent edges are used, so $p_i \in \P_n(G,S,T)$ for $i\leq m$.
  Suppose that $p_i \in \P_n(G,S,T)$ for some $i \in \{m, m+1, \dotsc, k-1\}$.
  We want to prove that $p_{i+1} = (\tilde \beta'', z'')$ also belongs to $\P_n(G,S,T)$.
  The edge $(p_i, p_{i+1})$ is colored of some color $s$.
  If $s\in T$ then $p_{i+1} \in \P_n(G,S,T)$ because no color of $S\setminus T$ is added.
  If $s\not\in T$, the $G$-action is trivial on $s$.
  Also, $(p_i,p_{i+1})$ is the only edge of $\gamma$ with color $s$.
  Therefore every other colored edge of $\gamma$ has a color $s' \neq s$, and cannot create elements colored by $s$ (because $s$ and $s'$ are in different orbits).
  Then $z'\,^{-1}(s) = z''\,^{-1}(s)$, and so $|z''\,^{-1}(s)| = |z'\,^{-1}(s)| \neq 1$ because $y \in \P_n(G,S,T)$.
  Therefore $p_{i+1} \in \P_n(G,S,T)$.
  By induction, the entire chain $\gamma$ is contained in $\P_n(G,S,T)$.
  
  \smallskip
  \textbf{Case 4: $y \neq \one$.}
  The proof works without changes in this case.
\end{proof}

\begin{remark}
  For $T=S$, Theorem \ref{thm:subposet-shellability} says that $\mu$ is an EL-labeling of $\hat \D_n(G,S)$.
\end{remark}

\begin{remark}
  Let $G = \Z_2$ act on $S = \{+, -, 0\}$ by exchanging $+$ and $-$.
  A computer check shows that the order complex of the poset $\hat P_3(G,S,\varnothing)$ is shellable.
  However, the edge labeling of Definition \ref{def:el-labeling} is not an EL-labeling of $\hat P_3(G,S,\varnothing)$.
\end{remark}

It seems difficult in general to derive an explicit formula for the number of decreasing maximal chains in $\hat \P_n(G,S,T)$.
This was done in \cite[Section 5]{delucchi2017shellability} for posets of layers of arrangements defined by root systems (i.e.\ with $G = \Z_2$ acting trivially on $S$, and $|T| = 0, 1$).

Lemma \ref{lemma:reduction-lemma} and Theorem \ref{thm:subposet-shellability} yield the following strengthening of Proposition \ref{prop:trivial-stabilizers}.

\begin{theorem}
  Let $n\geq 2$ and $|S| \geq 1$.
  Suppose that all the $G$-orbits in $S\setminus T$ either have cardinality 1, or have a trivial stabilizer.
  Then $\bar\P_n(G,S,T)$ is homotopy equivalent to a wedge of $d$-dimensional spheres with
  \[ d =
    \begin{cases}
      n-1 & \text{if $T \neq \varnothing$ or at least one $G$-orbit is trivial} \\
      n-2 & \text{otherwise}.
    \end{cases}
  \]
  \label{thm:homotopy-type-subposets}
\end{theorem}

\begin{proof}
  Remove all orbits of $S\setminus T$ with a trivial stabilizer, by a repeated application of Lemma \ref{lemma:reduction-lemma}.
  The remaining poset is shellable by Theorem \ref{thm:subposet-shellability}, and thus it is homotopy equivalent to a wedge of spheres.
\end{proof}

\section*{Acknowledgements}
I would like to thank Christin Bibby for the useful conversations about the topics of this paper.
I would also like to thank the anonymous referees for their comments and suggestions.
This work was supported by Scuola Normale Superiore, and by the Swiss National Science Foundation Professorship grant PP00P2\_179110/1.

\bibliographystyle{amsalpha-abbr}
\bibliography{bibliography}

\end{document}